\documentclass[oneside,10pt]{article}          
\usepackage[b5paper]{geometry}	    
\usepackage{amsfonts,amsmath,latexsym,amssymb} 
\usepackage{theorem}                
\usepackage{mathrsfs,upref}         
\usepackage{mathptmx}		    
\usepackage{jmi}	            

\usepackage{hyperref}
\DeclareMathOperator{\sech}{sech}
\DeclareMathOperator{\csch}{csch}

\newtheorem{theorem}{Theorem}           
\newtheorem{lemma}{Lemma}               
\newtheorem{corollary}{Corollary}

\theoremstyle{definition}

\newtheorem{example}{Example}

\begin{document}


\title[Short title]{On generalization of D'Aurizio-S\'andor trigonometric inequalities with a parameter}


\author{Li-Chang Hung and Pei-Ying Li}

\address{Li-Chang Hung, Department of Mathematics, National Taiwan University, Taipei, Taiwan\\
\email{lichang.hung@gmail.com}}

\address{Pei-Ying Li, Department of Finance, National Taiwan University, Taipei, Taiwan\\
\email{email of the Second Author}}

\CorrespondingAuthor{Li-Chang Hung}


\date{DD.MM.YYYY}                               




\begin{abstract}
In this work, we generalize the D'Aurizio-S\'andor inequalities (\cite{D'Aurizio,Sandor}) using an elementary approach. In particular, our approach provides an alternative proof of the D'Aurizio-S\'andor inequalities. Moreover, as an immediate consequence of the generalized D'Aurizio-S\'andor inequalities, we establish the D'Aurizio-S\'andor-type inequalities for hyperbolic functions. 
%

\end{abstract}

\maketitle



\section{Introduction}


Based on infinite product expansions and inequalities on series and the Riemann's zeta function, D'Aurizio (\cite{D'Aurizio}) proved the following inequality:
\begin{equation}\label{eqn: D'Aurizio ineq}
\frac{1-\displaystyle\frac{\cos x}{\cos \frac{x}{2}}}{x^2}<\frac{4}{\pi^2},
\end{equation}
where $x\in(0,\pi/2)$. Using an elementary approach, S\'andor (\cite{Sandor}) offered an alternative proof of \eqref{eqn: D'Aurizio ineq} by employing trigonometric inequalities and an auxiliary function. In the same paper, S\'andor also provided the converse to \eqref{eqn: D'Aurizio ineq}:
\begin{equation}
\frac{1-\displaystyle\frac{\cos x}{\cos \frac{x}{2}}}{x^2}>\frac{3}{8},
\end{equation}
where $x\in(0,\pi/2)$. In addition, S\'andor found the following analogous inequality \eqref{eqn: Sandor ineq}  holds true for the case of sine functions:
              
\begin{theorem}[D'Aurizio-S\'andor inequalities (\cite{D'Aurizio,Sandor})]\label{thm: Daurizio-Sandor ineq original}
The two double inequalities        
\begin{equation}
\frac{3}{8}<\frac{1-\displaystyle\frac{\cos x}{\cos \frac{x}{2}}}{x^2}<\frac{4}{\pi^2}
\end{equation}
and 
\begin{equation}\label{eqn: Sandor ineq}
\frac{4}{\pi^2}(2-\sqrt{2})<\frac{\displaystyle 2-\frac{\sin x}{\sin \frac{x}{2}}}{x^2}<\frac{1}{4}
\end{equation}
hold for any $x\in(0,\pi/2)$.
\end{theorem}

Throughout this paper, we denote $\frac{1-\frac{\cos x}{\cos \frac{x}{p}}}{x^2}$ and $\frac{p-\frac{\sin x}{\sin \frac{x}{p}}}{x^2}$ by $f_c(x)$ and $f_s(x)$, respectively:

\begin{align}
f_{p}^{c}(x)= & \frac{1-\displaystyle\frac{\cos x}{\cos \frac{x}{p}}}{x^2}, \\
f_{p}^{s}(x)= & \frac{p-\displaystyle\frac{\sin x}{\sin \frac{x}{p}}}{x^2}.
\end{align}

Our aim is to generalize the D'Aurizio-S\'andor inequalities for the case of $f_{p}^{c}(x)$ and $f_{p}^{s}(x)$ as follows:

\begin{theorem}[Generalized D'Aurizio-S\'andor inequalities]\label{thm: Daurizio-Sandor ineq generalized}
Let $0<x<\pi/2$. Then the two double inequalities
\begin{equation}\label{eqn: cos ineq p>2}
\frac{4}{\pi^2}<\frac{1-\displaystyle\frac{\cos x}{\cos \frac{x}{p}}}{x^2}<\frac{p^2-1}{2\,p^2}
\end{equation}
and 
\begin{equation}\label{eqn: sin ineq p>2}
\frac{4}{\pi^2}\left(p-\csc \left(\frac{\pi}{2p}\right)\right)<\frac{\displaystyle p-\frac{\sin x}{\sin \frac{x}{p}}}{x^2}<\frac{p^2-1}{6\,p}
\end{equation}
hold for $p=3, 4 , 5,\cdots$.  In particular, the double inequality \eqref{eqn: sin ineq p>2} remains true when $p=2$ while the double inequality \eqref{eqn: cos ineq p>2} is reversed when $p=2$.
\end{theorem}

The remainder of this paper is organized as follows. Section~\ref{sec: proof of the main results} is devoted to the proof of Theorem~\ref{thm: Daurizio-Sandor ineq generalized} and an alternative proof of Theorem~\ref{thm: Daurizio-Sandor ineq original}. In Section~\ref{sec: D'Aurizio-Sandor inequalities for hyperbolic functions}, we establish analogue of Theorem~\ref{thm: Daurizio-Sandor ineq generalized} for hyperbolic functions. As an application of Theorem~\ref{thm: Daurizio-Sandor ineq generalized}, we apply in Section~\ref{sec: application} inequality \eqref{eqn: sin ineq p>2} to the Chebyshev polynomials of the second kind and establish a trigonometric inequality. 


\section{Proof of the main results}\label{sec: proof of the main results}

At first we will prove the following lemma. The lemma provides expressions of the higher-order derivative  $\frac{d^2}{dx^2}(x^3\frac{d}{dx} f_{p}^{\triangle}(x))$ involving $f_{p}^{\triangle}(x)$ $(\triangle=c, s)$, which are helpful in proving Theorem~\ref{thm: Daurizio-Sandor ineq generalized}. We note that the sign of $\frac{d^2}{dx^2}(x^3\frac{d}{dx} f_{p}^{\triangle}(x))$ plays a crucial role in proving Theorem~\ref{thm: Daurizio-Sandor ineq generalized}.

\begin{lemma}\label{lem: lots of f'''}
Let $0<x<\pi/2$ and $k=1,2,3,\cdots$. Then when $p\in\mathbb{R}$ and $p\neq0$, we have
\begin{description}
  \item[$(i)$] 
\begin{align}\label{eqn: f''' cos} \nonumber
\frac{d^2}{dx^2}& \left(x^3\frac{d}{dx} f_{p}^{c}(x)\right)
=-\frac{x\,\csc ^4\left(\frac{x}{p}\right)}{8\,p^3}
\bigg(   
\left(p+1\right)^3 
   \sin \Big(x-\frac{3 x}{p}\Big)+\left(p-1\right)^3 \sin 
   \Big(x+\frac{3 x}{p}\Big)\\ 
+& \left(3 p^3+3 p^2-15 p-23\right)
   \sin \Big(x-\frac{x}{p}\Big)+\left(3 p^3-3 p^2-15 p+23\right) 
   \sin\Big(x+\frac{x}{p}\Big)\bigg);
\end{align} 
  \item[$(ii)$] 
\begin{align}\label{eqn: f''' sin} \nonumber
\frac{d^2}{dx^2}& \left(x^3\frac{d}{dx} f_{p}^{s}(x)\right)
=\frac{x\,\csc ^4\left(\frac{x}{p}\right)}{8\,p^3}
\bigg(   
\left(p+1\right)^3 
   \sin \Big(x-\frac{3 x}{p}\Big)-\left(p-1\right)^3 \sin 
   \Big(x+\frac{3 x}{p}\Big)\\ 
+& \left(-3 p^3-3 p^2+15 p+23\right)
   \sin \Big(x-\frac{x}{p}\Big)+\left(3 p^3-3 p^2-15 p+23\right) 
   \sin\Big(x+\frac{x}{p}\Big)\bigg). 
\end{align} 
In particular,
  \item[$(iii)$] when $p=2\,k$, 
\begin{equation}\label{eqn: 2k sin}
\frac{d^2}{dx^2} \left(x^3\frac{d}{dx} f_{p}^{s}(x)\right)=-\frac{x}{4\,k^3}\sum_{j=0}^{k-1} (2\,j+1)^3\,\sin\left(\frac{2\,j+1}{2\,k}\,x\right);
\end{equation}
  \item[$(iv)$] when $p=2\,k+1$, 
\begin{align}
\label{eqn: 2k+1 cos}
\frac{d^2}{dx^2} \left(x^3\frac{d}{dx} f_{p}^{c}(x)\right)=&-\frac{16\,x}{(2\,k+1)^3}\sum_{j=1}^{k} j^3\,\sin\left(\frac{2\,j}{2\,k+1}\,x\right)\,(-1)^{j-1}, \\
\label{eqn: 2k+1 sin}
\frac{d^2}{dx^2} \left(x^3\frac{d}{dx} f_{p}^{s}(x)\right)=&-\frac{16\,x}{(2\,k+1)^3}\sum_{j=1}^{k} j^3\,\sin\left(\frac{2\,j}{2\,k+1}\,x\right).
\end{align}
 \item[$(v)$] For $\triangle=c, s$ and $p\in\mathbb{R}\setminus\{0\}$, 
\begin{equation}
 \displaystyle \lim_{x\to0} \frac{d}{dx}\left(x^3\frac{d}{dx} f_{p}^{\triangle}(x)\right)=\lim_{x\to0} x^3\frac{d}{dx} f_{p}^{\triangle}(x)=0.
\end{equation}
\end{description}
\end{lemma}

\begin{proof}
$(i)$, $(ii)$ and $(v)$ follows directly from calculations using elementary Calculus. In particular, trigonometric addition formulas are used in proving $(i)$ and $(ii)$. To prove \eqref{eqn: 2k sin}, we claim
\begin{equation}
-\frac{x}{4\,k^3}\sum_{j=0}^{k-1} (2\,j+1)^3\,\sin\left(\frac{2\,j+1}{2\,k}\,x\right)=-\,x\,\frac{d^3}{dx^3} \left( \frac{\sin x}{\sin\left(\frac{x}{2\,k}\right)}\right).
\end{equation}
Indeed, we rewrite
\begin{equation}\label{eqn: trans f''' p=2k}
\frac{1}{4\,k^3}\sum_{j=0}^{k-1} (2\,j+1)^3\,\sin\left(\frac{2\,j+1}{2\,k}\,x\right)
= 2\frac{d^3}{dx^3} \left(\sum_{j=0}^{k-1}\cos\left(\frac{2\,j+1}{2\,k}\,x\right)\right).  
\end{equation}
On the other hand, making use of Euler's formula $e^{i\,z}=\cos z+i\,\sin z$ leads to an alternative expression of the left-hand side of \eqref{eqn: trans f''' p=2k}:
\begin{align}
\label{}
\sum_{j=0}^{k-1}\cos\left(\frac{2\,j+1}{2\,k}\,x\right)
&= \sum_{j=0}^{k-1} \displaystyle\Re\left\{ e^{i(\frac{x}{2\,k}+\frac{x}{k}\,j)}\right\}
  = \Re\left\{e^{i\frac{x}{2\,k}}\sum_{j=0}^{k-1} \displaystyle \left(e^{i\frac{x}{k}}\right)^j\right\}   \\
&= \Re\left\{e^{i\frac{x}{2\,k}}\displaystyle\frac{1-e^{i\,x}}{1-e^{i\frac{x}{k}}} \right\}
  = \Re\left\{e^{i\frac{x}{2\,k}}\displaystyle\frac{e^{\frac{i\,x}{2}}(e^{-\frac{i\,x}{2}}-e^{\frac{i\,x}{2}})}{e^{\frac{i\,x}{2\,k}}(e^{-\frac{i\,x}{2\,k}}-e^{\frac{i\,x}{2\,k}})} \right\}\\
&= \Re\left\{ e^{\frac{i\,x}{2}}\,\frac{\sin \left(\frac{x}{2}\right)}{\sin \left(\frac{x}{2\,k}\right)}\right\}
  =\cos \left(\frac{x}{2}\right)\frac{\sin \left(\frac{x}{2}\right)}{\sin \left(\frac{x}{2\,k}\right)}=\frac{\sin x}{2\,\sin \left(\frac{x}{2\,k}\right)},
\end{align}
where $\Re\left\{z\right\}$ is the real part of $z$ and $i=\sqrt{-1}$. Now it suffices to show
\begin{equation}
\frac{d^2}{dx^2} \left(x^3\frac{d}{dx} f_{p}^{s}(x)\right)=-x\frac{d^3}{dx^3} \left( \frac{\sin x}{\sin\left(\frac{x}{2\,k}\right)}\right).
\end{equation}
Using \eqref{eqn: f''' sin} in $(ii)$, this can be achieved by straightforward calculations. Thus $(iii)$ is true. The proof of $(iv)$ is similar, and we omit the details. We complete the proof of Lemma~\ref{lem: lots of f'''}.\qed


\end{proof}

We provide here an alternative proof of the two double inequalities in Theorem~\ref{thm: Daurizio-Sandor ineq original}.

\begin{proof}[Proof of Theorem~\ref{thm: Daurizio-Sandor ineq original}]
 To this end, we show that for $x\in(0,\pi/2)$, $f_{2}^{c}(x)=\frac{1-\frac{\cos x}{\cos \frac{x}{2}}}{x^2}$ is strictly increasing while $f_{2}^{s}(x)=\frac{2-\frac{\sin x}{\sin \frac{x}{2}}}{x^2}$ is strictly decreasing. These lead to the desired inequalities since it is easy to see that 
\begin{alignat}{4}
\lim_{x\to0}f_{2}^{c}(x) & =\frac{3}{8}, 
& \quad
\lim_{x\to\pi/2}f_{2}^{c}(x) & =\frac{4}{\pi^2}, 
& \quad
\\
\lim_{x\to0}f_{2}^{s}(x) & =\frac{1}{4}, 
& \quad
\lim_{x\to\pi/2}f_{2}^{s}(x) & =\frac{4}{\pi^2}(2-\sqrt{2}). 
& \quad
\end{alignat}
To see $f_{2}^{c}(x)$ is strictly increasing, we employ \eqref{eqn: f''' cos} in Lemma~\ref{lem: lots of f'''} to obtain
\begin{align}\label{eqn: f''' lemma1}\nonumber
\frac{d^2}{dx^2}\left(x^3\frac{d}{dx} f_{2}^{c}(x)\right)
&=-\frac{x}{64} \sec^4\left(\frac{x}{2}\right) \left(-44 \sin\left(\frac{x}{2}\right)+5 \sin \left(\frac{3 x}{2}\right)+\sin \left(\frac{5 x}{2}\right)\right)     \\
&=-\frac{x}{16} \sec^4\left(\frac{x}{2}\right)
\sin \left(\frac{x}{2}\right) (\cos x-2)(\cos x+5)>0. 
\end{align}
As $\lim_{x\to0} \frac{d}{dx}\left(x^3\frac{d}{dx} f_{2}^{c}(x)\right)=0$, it follows that $\frac{d}{dx}\left(x^3\frac{d}{dx} f_{2}^{c}(x)\right)>0$. We are led to $x^3\frac{d}{dx} f_{2}^{c}(x)>0$ or $\frac{d}{dx} f_{2}^{c}(x)>0$ since $\lim_{x\to0} \left(x^3\frac{d}{dx} f_{2}^{c}(x)\right)=0$. This that shows $f_{2}^{c}(x)$ is strictly increasing.

By using \eqref{eqn: 2k sin} in Lemma~\ref{lem: lots of f'''}, we have
\begin{equation}
\frac{d^2}{dx^2}\left(x^3\frac{d}{dx} f_{2}^{s}(x)\right)=-\frac{x}{4}\,\sin \left(\frac{x}{2}\right)<0,
\end{equation}
from which we infer that $\frac{d}{dx}\left(x^3\frac{d}{dx} f_{2}^{s}(x)\right)<0$ since $\lim_{x\to0} \frac{d}{dx}\left(x^3\frac{d}{dx} f_{2}^{s}(x)\right)=0$ by $(v)$ of Lemma~\ref{lem: lots of f'''}. Then
\begin{equation} 
\displaystyle \frac{d}{dx}\left(x^3\frac{d}{dx} f_{2}^{s}(x)\right)<0
\end{equation}
together with the fact $\lim_{x\to0} \left(x^3\frac{d}{dx} f_{2}^{s}(x)\right)=0$ from $(v)$ of Lemma~\ref{lem: lots of f'''} yields $x^3\frac{d}{dx} f_{2}^{s}(x)<0$ or $\frac{d}{dx} f_{2}^{s}(x)<0$. Thus we have shown that $f_{2}^{s}(x)$ is strictly decreasing. This completes the proof of the theorem.\qed
\end{proof}

We are now in the position to give the proof of Theorem~\ref{thm: Daurizio-Sandor ineq generalized}.

\begin{proof}[Proof of Theorem~\ref{thm: Daurizio-Sandor ineq generalized}]
The proof of the case when $p=2$ has been given in Theorem~\ref{thm: Daurizio-Sandor ineq generalized}. For $p\ge3$, we prove the desired inequalities by showing that $\frac{d}{dx} f_{p}^{\triangle}(x)<0$ for $\triangle=c, s$. Due to $(i)$ of Lemma~\ref{lem: lots of f'''}, we see that $\frac{d^2}{dx^2}\left(x^3\frac{d}{dx} f_{p}^{c}(x)\right)<0$ for $p\ge3$. Instead of employing $(ii)$ of Lemma~\ref{lem: lots of f'''}, we use \eqref{eqn: 2k sin} and \eqref{eqn: 2k+1 sin} in Lemma~\ref{lem: lots of f'''} to conclude that $\frac{d^2}{dx^2}\left(x^3\frac{d}{dx} f_{p}^{s}(x)\right)<0$. Thus we have for $\triangle=c, s$,
\begin{equation}
\frac{d^2}{dx^2}\left(x^3\frac{d}{dx} f_{p}^{\triangle}(x)\right)<0.
\end{equation}
Because of the first vanishing limit in $(v)$ of Lemma~\ref{lem: lots of f'''}, it follows that 
\begin{equation}
\frac{d}{dx}\left(x^3\frac{d}{dx} f_{p}^{\triangle}(x)\right)<0,
\end{equation}
which, together with the fact that the second limit in $(v)$ of Lemma~\ref{lem: lots of f'''} vanishes, implies that $x^3\frac{d}{dx} f_{p}^{\triangle}(x)<0$ or $\frac{d}{dx} f_{p}^{\triangle}(x)<0$ for $\triangle=c, s$. It remains to find the following limits:
\begin{alignat}{4}
\lim_{x\to0}f_{p}^{c}(x) & =\frac{p^2-1}{2\,p^2}, 
& \quad
\lim_{x\to\pi/2}f_{p}^{c}(x) & =\frac{4}{\pi^2}, 
& \quad
\\
\lim_{x\to0}f_{p}^{s}(x) & =\frac{p^2-1}{6\,p}, 
& \quad
\lim_{x\to\pi/2}f_{p}^{s}(x) & =\frac{4}{\pi^2}\left(p-\csc \left(\frac{\pi}{2p}\right)\right). 
& \quad
\end{alignat}
We immediately have
\begin{equation}
\frac{4}{\pi^2}=\lim_{x\to\pi/2}f_{p}^{c}(x)<\frac{1-\displaystyle\frac{\cos x}{\cos \frac{x}{p}}}{x^2}<\lim_{x\to0}f_{p}^{c}(x)=\frac{p^2-1}{2\,p^2}
\end{equation}
and 
\begin{equation}
\frac{4}{\pi^2}\left(p-\csc \left(\frac{\pi}{2p}\right)\right)=\lim_{x\to\pi/2}f_{p}^{s}(x)<\frac{\displaystyle p-\frac{\sin x}{\sin \frac{x}{p}}}{x^2}<\lim_{x\to0}f_{p}^{s}(x)=\frac{p^2-1}{6\,p}.
\end{equation}
The proof is completed.\qed

\end{proof}

\section{Generalized D'Aurizio-S\'andor inequalities for hyperbolic functions}\label{sec: D'Aurizio-Sandor inequalities for hyperbolic functions}

In this section, we show an analogue of Theorem~\ref{thm: Daurizio-Sandor ineq generalized} for the case of hyperbolic functions holds true. Let
\begin{align}
h_{p}^{c}(x)& = \frac{1-\displaystyle\frac{\cosh x}{\cosh \frac{x}{p}}}{x^2}, \\
h_{p}^{s}(x)& = \frac{p-\displaystyle\frac{\sinh x}{\sinh \frac{x}{p}}}{x^2}.
\end{align}
Following the same arguments for proving Lemma~\ref{lem: lots of f'''}, it can be shown that Lemma~\ref{lem: lots of f'''} with $\cos x$, $\sin x$ and $f_{p}^{\triangle}(x)$ $(\triangle=c, s)$ replaced by $\cosh x$, $\sinh x$ and $h_{p}^{\triangle}(x)$ $(\triangle=c, s)$ respectively, remains true. It follows that we can prove $\frac{d}{dx} h_{p}^{\triangle}(x)<0$ for $\triangle=c, s$ as in the proof of Theorem~\ref{thm: Daurizio-Sandor ineq generalized}. It remains to calculate the following limits:
\begin{alignat}{4}
\lim_{x\to0}f_{p}^{c}(x) & =\frac{1-p^2}{2\,p^2}, 
& \quad
\lim_{x\to\pi/2}f_{p}^{c}(x) & =\frac{4}{\pi^2}\left(1-\cosh \left(\frac{\pi }{2}\right)\sech\left(\frac{\pi }{2 p}\right)\right), 
& \quad
\\
\lim_{x\to0}f_{p}^{s}(x) & =\frac{1-p^2}{6\,p}, 
& \quad
\lim_{x\to\pi/2}f_{p}^{s}(x) & =\frac{4}{\pi ^2}\left(p-\sinh \left(\frac{\pi }{2}\right)\csch\left(\frac{\pi }{2p}\right)\right). 
& \quad
\end{alignat}
Thus, we have the following analogue of Theorem~\ref{thm: Daurizio-Sandor ineq generalized} for $\cosh x$ and $\sinh x$. 
\begin{theorem}\label{thm: Daurizio-Sandor ineq generalized hyperbolic fcn}
Let $0<x<\pi/2$. Then the two double inequalities
\begin{equation}\label{eqn: cosh ineq p>2}
\frac{4}{\pi^2}\left(1-\cosh \left(\frac{\pi }{2}\right)\sech\left(\frac{\pi }{2 p}\right)\right)<\frac{1-\displaystyle\frac{\cosh x}{\cosh \frac{x}{p}}}{x^2}<\frac{1-p^2}{2\,p^2}
\end{equation}
and 
\begin{equation}\label{eqn: sinh ineq p>2}
\frac{4}{\pi ^2}\left(p-\sinh \left(\frac{\pi }{2}\right)\csch\left(\frac{\pi }{2p}\right)\right)<\frac{\displaystyle p-\frac{\sinh x}{\sinh \frac{x}{p}}}{x^2}<\frac{1-p^2}{6\,p}
\end{equation}
hold for $p=3, 4 , 5,\cdots$. In particular, the double inequality \eqref{eqn: cosh ineq p>2} is reversed when $p=2$ while the double inequality \eqref{eqn: sinh ineq p>2} remains true when $p=2$.
\end{theorem}


\section{Application of the generalized D'Aurizio-S\'andor inequalities to the Chebyshev polynomials of the second kinds}\label{sec: application}

The first few Chebyshev polynomials of the second kind $U_n(x)$ $(n=0,1,2,\cdots)$ are (\cite{AS2,Chebyshev})
\begin{align}
\label{eqn: Chebyshev poly n=0}
U_0(x)&=1, \\
\label{eqn: Chebyshev poly n=1}
U_1(x)&=2\,x, \\
\label{eqn: Chebyshev poly n=1}
U_2(x)&=4\,x^2-1, \\
\label{eqn: Chebyshev poly n=1}
U_3(x)&=8\,x^3-4\,x, \\
\label{eqn: Chebyshev poly n=1}
U_4(x)&=16\,x^4-12\,x^2+1, \\
\label{eqn: Chebyshev poly n=1}
U_5(x)&=32\,x^5-32\,x^3+6\,x, \\
\label{eqn: Chebyshev poly n=1}
U_6(x)&=64\,x^6-80\,x^4+24\,x^2-1.
\end{align}
In this section, we apply Theorem~\ref{thm: Daurizio-Sandor ineq generalized} to $U_n(x)$ with $x=\cos \theta$. By means of the formula $U_{n}(\cos \theta)=\frac{\sin ((n+1)\,\theta)}{\sin \theta}$, we obtain the following corollary.

\begin{corollary}\label{cor: Daurizio-Sandor ineq applied to Chebyshev poly}
Let $y\in(0,\frac{\pi}{2\,p})$. The double inequality
\begin{equation}
\frac{p}{6}\left((1-p^2 )y^2+6\right)<U_{p-1}(\cos y)<p-\frac{4}{\pi^2}\left(p-\csc \left(\frac{\pi}{2p}\right)\right)\,p\,y^2
\end{equation}
holds for $p=2, 3, 4 , 5,\cdots$. 
\end{corollary}

\begin{proof}
The double inequality \eqref{eqn: sin ineq p>2} in Theorem~\ref{thm: Daurizio-Sandor ineq generalized} can be written as
\begin{equation}
p-\frac{p^2-1}{6\,p}\,x^2<\frac{\sin x}{\sin \frac{x}{p}}<p-\frac{4}{\pi^2}\left(p-\csc \left(\frac{\pi}{2p}\right)\right)\,x^2, \quad x\in(0,\pi/2).
\end{equation}
Letting $x/p=y$, we have
\begin{equation}
\frac{p}{6}\left((1-p^2 )y^2+6\right)<\frac{\sin (p\,y)}{\sin y}<p-\frac{4}{\pi^2}\left(p-\csc \left(\frac{\pi}{2p}\right)\right)\,p\,y^2, \quad y\in(0,\frac{\pi}{2\,p}).
\end{equation}
Since $\frac{\sin (p\,y)}{\sin y}=U_{p-1}(\cos y)$, the proof is completed.\qed
\end{proof}

\begin{example}
Letting $p=7$ in Corollary~\ref{cor: Daurizio-Sandor ineq applied to Chebyshev poly} results in the following inequality
\begin{equation}
7-56\,y^2<64\,\cos ^6 y-80\,\cos ^4 y+24\,\cos ^2 y-1<7-\frac{196 \left(7-\csc \left(\frac{\pi}{14}\right)\right)}{\pi ^2}\,y^2,
\end{equation}
where $y\in(0,\frac{\pi}{14})\approx (0,0.2244)$ and $\frac{196 \left(7-\csc \left(\frac{\pi}{14}\right)\right)}{\pi ^2}\approx 49.7673$.
\end{example}

\textit{Acknowledgements}. The authors wish to express sincere gratitude to Tom Mollee for his careful reading of the manuscript and valuable suggestions to improve the readability of the paper. Thanks are also due to Chiun-Chuan Chen and Mach Nguyet Minh for the fruitful discussions. The authors are grateful to the anonymous referee for many helpful comments and valuable suggestions on this paper.


\Refs

\bibitem{AS2}
        \by M. Abramowitz and I. A. Stegun (Eds)
        \book Handbook of Mathematical Functions with Formulas, Graphs, and Mathematical Tables
        \publ National Bureau of Standards, Applied Mathematics Series {\bf 55}, 9th printing
        \publaddr Washington 
        \year 1970
\endref

\bibitem{D'Aurizio}
        \by J. D'Aurizio
        \paper Refinements of the {S}hafer-{F}ink inequality of arbitrary
              uniform precision
        \jour Math. Inequal. Appl.
        \vol 17
        \issue 4
        \yr 2014
        \pages 1487--1498
\endref

\bibitem{Chebyshev}
        \by T. J. Rivlin
        \book Chebyshev Polynomials
        \publ Wiley
        \publaddr New York 
        \year 1990
\endref

\bibitem{Sandor}
        \by J. S\'andor
        \paper On D'Aurizio's trigonometric inequality
        \jour J. Math. Inequal.
        \vol 10
        \issue 3
        \yr 2016
        \pages 885--888
\endref

\endRefs

\end{document}